\documentclass[10pt,reqno]{amsart}
\usepackage[utf8]{inputenc}

\usepackage{microtype}

\usepackage{amsmath}	
\usepackage{amsthm}		
\usepackage{amssymb}
\usepackage{graphicx}
\usepackage{hyperref}
\usepackage{color}
\usepackage{enumerate}	

\usepackage{geometry}
\usepackage{pgf,tikz}

\DeclareMathOperator{\reg}{reg}
\DeclareMathOperator{\bei}{J}
\DeclareMathOperator{\NumOfMaxCl}{c}

\DeclareMathOperator{\SizeOfMaxStabS}{\alpha}
\DeclareMathOperator{\NumOfMaxStabS}{s}
\DeclareMathOperator{\cone}{cone}
\DeclareMathOperator{\N}{\mathbb{N}}

\DeclareMathOperator{\Hilbert}{Hilb}

\newcommand{\compl}[1]{\overline{#1}}

\theoremstyle{definition}
\newtheorem{definition}{Definition}[section]

\newtheorem{example}[definition]{Example}
\newtheorem{remark}[definition]{Remark}
\theoremstyle{plain}
\newtheorem{lemma}[definition]{Lemma}
\newtheorem{proposition}[definition]{Proposition}
\newtheorem{theorem}[definition]{Theorem}
\newtheorem{corollary}[definition]{Corollary}

\newcommand{\kk}{\Bbbk}
\newcommand{\NN}{\mathbb{N}}
\renewcommand\>{\rangle}

\renewcommand{\subset}{\subseteq}
\newcommand{\chara}{\operatorname{char} }

\allowdisplaybreaks[1]

\definecolor{darkblue}{RGB}{0,0,160}
\hypersetup{
	colorlinks,%
	citecolor=black,%
	filecolor=black,%
	linkcolor=darkblue,%
	urlcolor=darkblue
}

\begin{document}

\title{Binomial edge ideals of cographs}

\author[T.~Kahle]{Thomas Kahle}
\address[T.~Kahle]{Fakultät für Mathematik, OvGU Magdeburg, Magdeburg, Germany}
\urladdr{\url{http://www.thomas-kahle.de}}

\author[J.~Krüsemann]{Jonas Krüsemann}
\address[J.~Krüsemann]{Rail Management Consultants GmbH, Hannover, Germany}
\email{jonas.kruesemann@t-online.de}

\keywords{binomial edge ideal, undirected graph, Castelnuovo--Mumford regularity, Betti numbers}
\subjclass[2010]{Primary 05E40; Secondary 13D02, 05C99, 13P20}


\thanks{Supported by the Deutsche Forschungsgemeinschaft (314838170, GRK 2297
MathCoRe).}

\begin{abstract}
  We determine the Castelnuovo--Mumford regularity of binomial edge ideals of
  complement reducible graphs (cographs).  For cographs with $n$ vertices the
  maximum regularity grows as~$2n/3$.  We also bound the regularity by graph
  theoretic invariants and construct a family of counterexamples to a conjecture
  of Hibi and Matsuda.
\end{abstract}

\maketitle

\section{Introduction}
Let $G = ([n],E)$ be a simple undirected graph on the vertex set
$[n]= \{1,\dots, n\}$.  Let $X = \left(
\begin{smallmatrix}
  x_{1}& \cdots & x_{n}\\
  y_{1}& \cdots & y_{n}
\end{smallmatrix}\right)
$ be a generic $2\times n$ matrix and $S = \kk[\begin{smallmatrix}
  x_{1}& \cdots & x_{n}\\
  y_{1}& \cdots & y_{n}
\end{smallmatrix}
]$ the polynomial ring whose indeterminates are the entries of~$X$ and with
coefficients in a field~$\kk$.  The \emph{binomial edge ideal of $G$} is
$\bei_{G} = \<x_{i}y_{j} - y_{i}x_{j} : \{i,j\}\in E\> \subset S$, the ideal of
$2\times 2$ minors indexed by the edges of the graph.  Since their inception in
\cite{herzogBEIorig,ohtani2011graphs}, connecting combinatorial properties of
$G$ with algebraic properties of $\bei_{G}$ or $S/\bei_{G}$ has been a popular
activity.  Particular attention has been paid to the minimal free resolution of
$S/\bei_{G}$ as a standard $\NN$-graded $S$-module \cite{EneZarojanu2013,
kiani2016castelnuovo}.  The data of a minimal free resolution is encoded in its
graded Betti numbers
$\beta_{i,j}(S/\bei_{G}) = \dim_\kk\text{Tor}_i(S/\bei_{G},\kk)_j$.  An interesting
invariant is the highest degree appearing in the resolution, the
Castelnuovo--Mumford regularity
$\reg (S/\bei_{G})=\max\{j-i:\beta_{ij} (S/\bei_{G})\neq 0\}$.  It is a
complexity measure as low regularity implies favorable properties like vanishing
of local cohomology.  Binomial edge ideals have square-free initial ideals by
\cite[Theorem~2.1]{herzogBEIorig} and, using~\cite{conca2018square}, this
implies that the extremal Betti numbers and regularity can also be derived from
those initial ideals.  In this paper we rely on recursive constructions of
graphs rather than Gröbner deformations.

At the time of writing it is unknown if the regularity of $S/\bei_{G}$ depends
on the characteristic $\chara(\kk)$ of the coefficient field.  Indication that it
is indeed independent comes, for example, from a purely combinatorial
description of the linear strand of the determinantal facet ideals
in~\cite{linstrand}.  Here we use only combinatorial constructions that are
independent of $\chara(\kk)$, based on small graphs for which the minimal free
resolutions are independent of~$\kk$.  Our starting point are the following bounds
due to Matsuda and Murai~\cite{matsuda2013regularity} which are valid
independent of~$\chara(\kk)$.
\begin{theorem}\label{theorem:matsudaMurai}
  Let $\ell$ be the maximum length of an induced path in a
  graph~$G$. Then
  \begin{equation*}
    \ell \leq \reg(S/\bei_{G}) \leq n-1.
  \end{equation*}
\end{theorem}
Our aim is to investigate families of graphs for which the lower bound is
constant.  We study the family of graphs with no induced path of length~$3$.
These are the \emph{complement reducible graphs (cographs)}.  They have been
characterized in \cite[Theorem~2]{CorneilLerchsBurlingham1981} as graphs such
that for every connected induced subgraph with at least two vertices, the
complement of that subgraph is disconnected.  Cographs are hereditary in the
sense that every induced subgraph of a cograph is a cograph
\cite[Lemma~1]{CorneilLerchsBurlingham1981}.  The graphs $G$ for which
$\reg(S/\bei_{G}) \le 2$ are all cographs by the ($\chara(\kk)$-independent)
characterization in \cite[Theorem~3.2]{madani2018binomial}.

One can quickly find that among (connected) cographs arbitrarily high regularity
is possible (Proposition~\ref{prop:cographsHighReg}), but when the regularity is
considered as a function of the number of vertices of the cograph, there is a
stricter upper bound than that in Theorem~\ref{theorem:matsudaMurai}.  Our
Theorem~\ref{theorem:cographs:maxReg} shows that for cographs the regularity is
essentially bounded by~$2n/3$.  The experimental results and also the proof
methods leading to the $2n/3$ bound lead us to study regularity bounds in terms
of other graph invariants.  Theorem~\ref{thm:boundByInvariants} bounds
$\reg(S/\bei_{G})$ by the independence number $\SizeOfMaxStabS(G)$ and the
number of maximal independent sets $\NumOfMaxStabS(G)$.  This stands in the
context of a general upper bound of $\reg(S/\bei_{G})$ by the number of maximal
cliques~$\NumOfMaxCl(G)$, shown by Malayeri, Madani and
Kiani~\cite{malayeri2021proof}.  Interestingly, for cographs the mentioned
invariants are computable in linear time, while in general they are hard to
compute.

The investigations for this paper started with a large experiment in which we
tabulated properties of binomial edge ideals and the corresponding graphs for
many small graphs.  To this end we developed a database of algebraic properties
of~$S/\bei_{G}$ and the necessary tools to extend the database.  Our source code
is available at
\begin{center}
  \url{https://github.com/kruesemann/graph_ideals}.
\end{center}
In the course of this work we found a counterexample to a conjecture of Hibi and
Matsuda.  We present this in Section~\ref{s:HibiMatsuda}, which leaves the class
of cographs.

\subsection*{Notation} All graphs in this paper are simple and undirected,
meaning that they have only undirected edges with no loops and no multiple edges
between two fixed vertices.  We often consider binomial edge ideals $\bei_{G}$
and $\bei_{H}$ of graphs on different vertex sets $V(G)\neq V(H)$.  In this case
we can consider both ideals in a bigger ring with variables corresponding to
$V(G) \cup V(H)$.  Embedding the ideals in a ring with extra variables has no
effect on the invariants under consideration and we still write $\bei_{G}$
and~$\bei_{H}$ independent of the ambient ring.  By \emph{the regularity of $G$}
we mean $\reg (S/\bei_{G})$.  The \emph{complement} $\compl{G}$ of a graph $G$
is a graph on the same vertex set, but with the edge set
$E(\compl{G}) = \{(i,j) : i\neq j, (i,j)\notin E(G)\}$.  The \emph{path of
length $\ell-1$} has $\ell$ vertices and is denoted by~$P_{\ell}$.

\section{Regularity for cographs}

Our first aim is to show that $\reg(S/\bei_{G})$ can take arbitrarily large
values, even if $G$ is restricted to connected cographs and the lower bound of
Theorem~\ref{theorem:matsudaMurai} does not apply.  If one allows disconnected
graphs, this follows from the simple observation that regularity is additive
under the disjoint union $G \sqcup H$ of two graphs:
\[
  \reg(S/\bei_{G \sqcup H}) = \reg(S/\bei_{G}) + \reg(S/\bei_{H}).
\]
As noted in \cite[proof of Theorem~2.2]{EneZarojanu2013}, this follows from the
fact that $\bei_{G}$ and $\bei_{H}$ use disjoint sets of variables.  In
particular this is true independent of~$\kk$.
To see that arbitrary regularity is possible for connected cographs, we employ
the \emph{join} of two simple undirected graphs $G$ and $H$ which is
\begin{equation*}
  G \ast H = (V(G) \sqcup V(H), E(G) \cup E(H) \cup \{\{v,w\} : v \in V(G), w \in V(H)\}).
\end{equation*}
Here and in the following, $V(G)$ and $E(G)$ denote, respectively, the vertex
set and the edge set of an undirected simple graph~$G$.  The join also behaves
nicely with regularity as shown by Kiani and
Madani~\cite[Theorem~2.1]{madani2018binomial}.

\begin{theorem}\label{theorem:join:reg}
  Let $G$ and $H$ be simple, undirected graphs and not both complete. Then,
  independent of $\chara(\kk)$,
  \begin{equation*}
    \reg(S/\bei_{G \ast H}) = \max\{\reg(S/\bei_{G}), \reg(S/\bei_{H}), 2\}.
  \end{equation*}
\end{theorem}

The join of two cographs is a cograph as the path $P_{4}$ of length 3
is not a join itself.  It follows that the regularity can be made
arbitrarily high by forming \emph{cones}, that is joins with single
vertex graphs.

\begin{proposition}\label{prop:cographsHighReg}
  For any $r\ge 1$ there is a connected cograph with $\reg(S/\bei_{G})=r$.  If
  $r$ is even, then there is such a graph with $\frac{3}{2}r + 1$ vertices.  If
  $r$ is odd, then there is such a graph with $\frac{3}{2}(r+3)$ vertices.
\end{proposition}
\begin{proof}
  For $r=1$ a single edge suffices.  If $r = 2a$ is even, let $G$ be a disjoint
  union of $a$ copies of~$P_{3}$.  If $3 \le r=2a+1$ is odd, then let $G$ be a
  disjoint union of $a$ copies of $P_{3}$ and a single edge~$P_{2}$.  In both
  cases $\reg (S/\bei_{G}) = r$.  Let $\cone(G)$ be the join of $G$ with a
  single vertex graph.  Theorem~\ref{theorem:join:reg} implies that,
  $\reg(S/\bei_{\cone(G)}) = \reg(S/\bei_{G})$ and thus $\cone(G)$ is a
  connected cograph with regularity~$r$.
\end{proof}

With the lower bound settled, we aim for a stricter upper bound on the
regularity of cographs.  For this we employ the original definition of
complement reducible graphs: they are constructed recursively by
taking complements and disjoint unions of cographs, starting from the
single-vertex graph.  As a result, there is an unusual
one-to-one-relationship between connected and disconnected cographs of
the same order.

\begin{lemma}\label{lemma:cographs:complement}
  Let $G$ be a cograph with at least $2$ vertices. Then $G$ is
  connected if and only if $\compl{G}$ is disconnected.
\end{lemma}
\begin{proof}
  Let $\compl{G}$ be disconnected and $v,w \in V(G)$.  If $v$ and $w$ are in
  different connected components of $\compl{G}$, then $\{v,w\} \in E(G)$, so
  there is a path from $v$ to $w$ in~$G$.
  If $v$ and $w$ are in the same connected component of~$\compl{G}$, then there
  exists another vertex $u \in V(G)$ in a different connected component
  of~$\compl{G}$.  In particular, $\{v,u\}, \{w,u\} \notin E(\compl{G})$ and so
  $\{v,u\}, \{w,u\} \in E(G)$, so $(v,u,w)$ is a path from $v$ to $w$ in $G$.
  Thus $G$ is connected.
	
  The other implication follows from the
  stronger~\cite[Theorem~2]{CorneilLerchsBurlingham1981}: every induced subgraph
  of a cograph with more than one vertex has disconnected complement.
\end{proof}

In Lemma~\ref{lemma:cographs:complement}, the implication $\compl{G}$
disconnected $\Rightarrow$ $G$ connected holds for any graph, not just
cographs.  In the same generality, the join of graphs can be expressed
using complement and disjoint union.

\begin{lemma}\label{lemma:join:complement}
  Let $G_1$ and $G_2$ be simple undirected graphs. Then
  $G_1 \ast G_2 = \compl{\compl{G_1} \sqcup \compl{G_2}}$.
\end{lemma}
\begin{proof}
  Let $V$ denote the common vertex set of both graphs.  Let $e = \{v,w\}$ where
  $v,w \in V$ are arbitrary vertices.  If $e \nsubseteq V(G_i)$ for
  $i \in \{1,2\}$, then $e$ is an edge in both
  $\compl{\compl{G_1} \sqcup \compl{G_2}}$ and $G_1 \ast G_2$.  If
  $e \subseteq V(G_i)$ for one $i \in \{1,2\}$, then
  $e \in E(\compl{\compl{G_1} \sqcup \compl{G_2}})$ if and only if $e$ is an
  edge in~$G_i$, which is the case if and only if $e \in E(G_1 \ast G_2)$.
\end{proof}

\begin{lemma}\label{lemma:cographs:join}
  A connected cograph $G$ is the join of induced subgraphs $G_1,\dots,G_m$,
  which are exactly the complements of the connected components of $\compl{G}$.
\end{lemma}
\begin{proof}
  By Lemma~\ref{lemma:cographs:complement}, $\compl{G}$ is disconnected.  Let
  $\compl{G_1},\dots,\compl{G_m}$ be the connected components of $\compl{G}$ and
  $G_1,\dots,G_m$ their complements.  The $G_{i}$ are induced subgraphs of $G$
  since their edge sets arise from that of $G$ by complementing twice.  Then, by
  Lemma~\ref{lemma:join:complement},
  $G = \compl{\compl{G_1} \sqcup \dots \sqcup \compl{G_m}} = G_1 \ast \dots \ast
  G_m$ as both $\sqcup$ and $\ast$ are associative.
\end{proof}
Using Lemma~\ref{lemma:cographs:join}, we can assume that any connected cograph
$G$ is written as a join of the complements of the connected components of its
complement.  Since any induced subgraph of a cograph is a cograph, the $G_{i}$
in the lemma are cographs too and by Lemma~\ref{lemma:cographs:complement} they
are disconnected or have only one vertex.

\begin{proposition}\label{proposition:cographs:regConnected}
  Let $G$ be a connected cograph that is not complete. Then
  \begin{equation*}
    \reg(S/\bei_{G}) = \max(\{2\} \cup \{\reg(S/\bei_{G_{i}}) : i \in \left[m\right]\}),
  \end{equation*}
  where $G = G_1 \ast \dots \ast G_m$ as in Lemma~\ref{lemma:cographs:join} and
  the $G_i$ are cographs that each are either disconnected or single vertices.
\end{proposition}
\begin{proof}
  Since $G$ is not complete, the $G_i$ cannot all be complete, so the equations
  follows from Lemma~\ref{lemma:cographs:join} and
  Theorem~\ref{theorem:join:reg}.
\end{proof}

We now have all the ingredients for a recursive computation of
$\reg(S/\bei_{G})$ for any cograph~$G$.  In the disconnected case, add the
regularities of all connected components.  In the connected case, compute the
maximum regularity of the complements $G_i$ of the connected components
$\compl{G_i}$ of~$\compl{G}$.  We use this to bound the maximum regularity.

\begin{theorem}\label{theorem:cographs:maxReg}
  Let $G$ be a cograph on $3k-a$ vertices, with $k \in \N$ and
  $a \in \{0,1,2\}$. Then
  \begin{equation*}
    \reg(S/\bei_{G}) \leq 2k - a.
  \end{equation*}
  If $G$ is connected, $k > 1$, and $a \in \{0,1\}$, then
  $\reg(S/\bei_{G}) \leq 2k - a - 1$.
\end{theorem}
\begin{proof}
  The proof is by induction over $k$.  The cographs with at most $3$ vertices
  are $K_1$, $K_2=P_1$, $\compl{K_2}$, $K_3$, $\compl{K_3}$, $P_2$ and
  $\compl{P_2}$.  Then $\reg(S/\bei_{G}) \leq 2 - a = \reg(S/\bei_{P_{3-a}})$.
	
  Now let $k>1$.  By Proposition~\ref{proposition:cographs:regConnected}, if $G$
  is connected, it either has regularity $2$ or there is a smaller, disconnected
  cograph with the same regularity.  So it can be assumed that $G$ is
  disconnected.  Let $H$ be a connected component of $G$ with $3k_H - a_H$
  vertices and $H' = G \setminus H$ have $3k_{H'} - a_{H'}$ vertices, where
  $k_H,k_{H'} \in \N$ and $a_H,a_{H'} \in \{0,1,2\}$.  Then both $H$ and $H'$
  have fewer vertices than $G$ and
  \begin{equation*}
      k_H + k_{H'} =
      \begin{cases}
	k &\text{if } a_H + a_{H'} = a, \\
      k + 1 &\text{if } a_H + a_{H'} = a + 3.
      \end{cases}
  \end{equation*}
  By induction $\reg(S/\bei_{H}) \leq 2k_H - a_H$ and $\reg(S/\bei_{H'})
  \leq 2k_{H'} - a_{H'}$, and 
  \begin{align*}
    \reg(S/\bei_{G}) &\leq 2k_H - a_H + 2k_{H'} - a_{H'} \\
		     &= \begin{cases}
		       2k - a & \text{if } a_H + a_{H'} = a, \\
		       2k - a - 1 & \text{if } a_H + a_{H'} = a + 3 \\
		     \end{cases} \\
		     &\leq 2k - a.
  \end{align*}
  If $G$ is connected and $k > 1$, the regularity is either $2$ or at most that
  of a disconnected cograph with fewer vertices.  So if $G$ has $3k - a$
  vertices, with $a \in \{0,1\}$, then $\reg(S/\bei_{G}) \leq 2k - a - 1$.  If
  $G$ has $3k - 2$ vertices, then $\reg(S/\bei_{G}) \leq 2(k - 1) = 2k - 2$.
\end{proof}

The following example shows that the bounds for disconnected cographs in
Theorem~\ref{theorem:cographs:maxReg} can be realized.
\begin{example}\label{ex:maxreg}
  Let $n = 3k-a \in\NN$ be a positive integer with $a\in\{0,1,2\}$.  If $a = 0$,
  let $G_{n}$ be a disjoint union of $k$ copies of $P_{3}$.  If $a = 1$, then
  let $G_{n}$ be a disjoint union of $k-1$ copies of $P_{3}$ and one of~$P_{2}$.
  If $k=2$, let $G_{n}$ be a disjoint union of two copies of $P_{2}$ and $k-2$
  copies of~$P_{3}$.  Since $P_{n}$ has regularity $n-1$ and regularity is
  additive under disjoint union, $G_{n}$ has regularity~$2k-a$.
\end{example}

To make connected examples, use the cone construction from
Proposition~\ref{prop:cographsHighReg}.  Even more, in two of three cases the
graphs from Example~\ref{ex:maxreg} are the only cographs of maximum regularity
as the following theorem shows.
\begin{theorem}\label{thm:Characterize}
  Let $G$ be a cograph with $3k - a$ vertices, where $k \in \N$ and
  $a \in \{0,1\}$. Then $\reg(S/\bei_{G}) = 2k - a$ if and only if $G$ is a
  disjoint union of $P_{3}$ and at most one~$P_{2}$.
\end{theorem}
\begin{proof}
  With Example~\ref{ex:maxreg} in place, just the only-if-direction remains.
  Let $G$ be a cograph.  By Theorem~\ref{theorem:cographs:maxReg}, there is
  nothing to prove if $G$ is connected, so we assume it is disconnected.  We
  first show that any connected component has at most $3$ vertices.  To this
  end, let $H$ be a connected component of $G$ with $3k_H - a_H$ vertices,
  $k_{H}>1$, and $a_H \in \{0,1,2\}$.  Let $H' = G \setminus H$ have
  $3k_{H'} - a_{H'}$ vertices with $a_{H'} \in \{0,1,2\}$.  Then
  \begin{equation*}
    k_H + k_{H'} =\begin{cases}
      k &\text{if } a_H + a_{H'} = a, \\
      k + 1 &\text{if } a_H + a_{H'} = a + 3.
    \end{cases}
  \end{equation*}
  If $a_{H}\neq 2$, then, by Theorem~\ref{theorem:cographs:maxReg},
  \begin{align*}
    \reg(S/\bei_{G}) &\leq 2k_H - a_H - 1 + 2k_{H'} - a_{H'} \\
		     &= \begin{cases}
		       2k - a - 1 & \text{if } a_H + a_{H'} = a, \\
		       2k - a - 2 & \text{if } a_H + a_{H'} = a + 3 \\
		     \end{cases} \\
		     &< 2k - a.
  \end{align*}
  If otherwise $a_H = 2$, then since $a \in \{0,1\}$, $G$ must have another
  connected component $H'$ with $3k_{H'} - a_{H'}$ vertices and
  $a_{H'} \in \{1,2\}$.  Let $H'' = G \setminus (H \sqcup H')$ with
  $3k_{H''} - a_{H''}$ vertices, where $k_{H''} \in \N_0$ and
  $a_{H''} \in \{0,1,2\}$.  Then 
  \begin{gather*}
    k_H + k_{H'} + k_{H''} =
    \begin{cases}
      k + 1 & \text{if } a_H + a_{H'} + a_{H''} = a + 3, \\
      k + 2 & \text{if } a_H + a_{H'} + a_{H''} = a + 6.
    \end{cases}
  \end{gather*}
  Therefore Theorem~\ref{theorem:cographs:maxReg} implies
  \begin{align*}
    \reg(S/\bei_{G}) &\leq 2k_H - a_H + 2k_{H'} - a_{H'} + 2k_{H''} - a_{H''} \\
		     &= \begin{cases}
		       2k - a - 1 & \text{if } a_H + a_{H'} + a_{H''} = a + 3, \\
		       2k - a - 2 & \text{if } a_H + a_{H'} + a_{H''} = a + 6 \\
		     \end{cases} \\
		     &< 2k - a.
  \end{align*}
  We have thus shown that if $G$ has a connected component with more than $3$
  vertices, it cannot have maximum regularity. Since the only cographs of
  maximum regularity with $3$ and $2$ vertices are $P_3$ and $P_2$ respectively,
  it follows that if $G$ has maximum regularity, then it is a disjoint union of
  $2$-paths $P_{3}$, single edges $P_{2}$ and isolated vertices.  We now analyze
  these cases separately for two possible values of $a$.
	
  Suppose first that $a=0$.  If $G$ has a connected component $H$ with one or
  two vertices, that is $3 - a_H$ vertices and $a_H \in \{1,2\}$, then $G$ must
  have another connected component $H'$ with $3 - a_{H'}$ vertices, where
  $a_{H'} \in \{1,2\}$.  With a similar computation as above we find that
  $\reg(S/\bei_{G}) < 2k$.  Therefore if $G$ has $3k$ vertices and maximal
  regularity, each connected component must have exactly $3$ vertices and be
  equal to~$P_{3}$.
	
  Finally, consider the case $a = 1$.  If $G$ has an isolated vertex, then $G$
  must have another connected component with fewer than $3$ vertices.  Then $G$
  cannot have maximal regularity as isolated vertices contribute no regularity
  and a disjoint union of an edge and a vertex has regularity 1.  If there are
  two isolated edges, then the subgraph on these $4$ vertices contributes
  regularity only $2$ and thus $G$ cannot have maximal regularity.
\end{proof}

\begin{remark}
  Graphs with $3k - 2$ vertices and maximal regularity, those in the $a=2$ case
  of Theorem~\ref{thm:Characterize}, do not have a simple characterization.  For
  example, the class contains cones over disjoint unions of $2$-paths as well as
  other types of joins and disjoint unions of joins, paths, and isolated
  vertices.
\end{remark}

The following corollary partially characterizes connected maximers of
regularity.
\begin{corollary}\label{cor:isCone}
  Let $G$ be a connected cograph with maximum regularity among connected
  cographs on $3k-a$ vertices with $k>1$ and $a \in \{0,2\}$, that is,
  $\reg(S/\bei_{G}) = 2k-1$ for $a=0$ and $\reg(S/\bei_{G}) = 2k-2$ for $a=2$.
  Then $G$ is a cone.
\end{corollary}
\begin{proof}
  Let $G$ be a connected cograph with $3k-a$ vertices and maximum regularity,
  where $k>1$ and $a\in\{0,2\}$.  By Lemma~\ref{lemma:cographs:join},
  $G = G_{1}\ast \dotsb \ast G_{m}$ where the $G_{i}$ are disconnected or single
  vertices.  Proposition~\ref{proposition:cographs:regConnected} shows that if
  $G$ is not a cone, that is, neither of the $G_{i}$ is a single vertex, then
  $\reg(S/\bei_{G}) = \reg(S/\bei_{G'})$ for some $G'$ which has at least two
  vertices fewer than~$G$.  So assume $G'$ has $3k-a-2$ vertices and apply
  Theorem~\ref{theorem:cographs:maxReg}.  If $a=0$, then $G'$ has $3k-2$
  vertices and thus regularity at most $2k-2$, but $G$ had maximum regularity
  $2k-1$, a contradiction.  If $a=2$, then $G'$ has $3(k-1)-1$ vertices and thus
  regularity at most $2(k-1)-1 = 2k-3$, but $G$ had maximum regularity $2k-2$,
  another contradiction.  By these contradictions, either $G$ is a cone after
  all, or its regularity was not maximal.
\end{proof}
Applying the reasoning in the proof to the $a=1$ case does not yield the
contradiction.  In this case, $G$ has $3k-1$ vertices and thus a regularity
bound of $2k-1-1 = 2k-2$ while $G'$ with its $3(k-1)$ vertices has regularity at
most~$2(k-1)$.


The results so far give a fairly clear picture of the asymptotic behaviour of
regularity in the class of cographs and individually for graphs with known
complement reducible decomposition.  If the exact structure of a cograph is
unknown, it can be useful to bound regularity using graph-theoretic invariants.
We consider here $\SizeOfMaxStabS(G)$ -- the size of the largest independent
set, $\NumOfMaxStabS(G)$ -- the number of maximal independent sets in $G$,
and~$\NumOfMaxCl(G)$ -- the number of maximal cliques of~$G$.  The recursive
construction of cographs yields bounds because these invariants satisfy simple
formulas under disjoint union and join:
\begin{align}
  \raggedright
  \NumOfMaxStabS(G \sqcup H) &= \NumOfMaxStabS(G)\NumOfMaxStabS(H), \label{equation:union:maxStable}\\
  \NumOfMaxStabS(G \ast H) &= \NumOfMaxStabS(G) + \NumOfMaxStabS(H), \label{equation:join:maxStable}\\
  \SizeOfMaxStabS(G \sqcup H) &= \SizeOfMaxStabS(G) + \SizeOfMaxStabS(H), \label{equation:union:stable}\\
  \SizeOfMaxStabS(G \ast H) &= \max\{\SizeOfMaxStabS(G), \SizeOfMaxStabS(H)\}. \label{equation:join:stable}
\end{align}

We find the following bounds which are independent of the number of vertices.
\begin{theorem}\label{thm:boundByInvariants}
  Let $G$ be a cograph. Then
  $\reg(S/\bei_{G}) \leq \min\{\NumOfMaxStabS(G), \SizeOfMaxStabS(G)\}$.
\end{theorem}
\begin{proof}

  The proof is by induction on the number of vertices of $G$.
  An isolated vertex has exactly one maximal independent set with one vertex and
  $\reg(S/\bei_{K_1}) = 0$, so the statement holds.
  Now let $G$ be any cograph.  If $G$ is connected, it is the join of smaller
  cographs $G_1,\dots,G_{m}$ and by induction
  \begin{equation*}
    \reg(S/\bei_{G_i}) \leq \min\{\NumOfMaxStabS(G_i), \SizeOfMaxStabS(G_i)\}
    \quad \text{ for all } i=1,\dots,m.
  \end{equation*}
  If $G$ is complete, both inequalities are trivial.  In the other case, it
  follows by Theorem~\ref{theorem:join:reg} and~\eqref{equation:join:maxStable}
  that
  \begin{equation*}
    \reg(S/\bei_{G}) = \max\{\{2\} \cup \{\reg(S/\bei_{G_i}) : i=1,\dots,m\}\}
    \leq \sum_{i=1}^{m} \NumOfMaxStabS(G_i) = \NumOfMaxStabS(G)
  \end{equation*}
  and, since $\SizeOfMaxStabS(G) = 1$ if and only if $G$ is complete,
  \eqref{equation:join:stable} leads to
  \begin{multline*}
    \reg(S/\bei_{G}) = \max\{\{2\} \cup \{\reg(S/\bei_{G_i}) : i=1,\dots,m\}\}
    \\
    \leq \max\{\SizeOfMaxStabS(G_i) : i=1,\dots,m\} = \SizeOfMaxStabS(G).
  \end{multline*}
  If $G$ is disconnected, it has connected components $G_1,\dots,G_{m}$ which by
  induction satisfy
  \begin{equation*}
    \reg(S/\bei_{G_i}) \leq \min\{\NumOfMaxStabS(G_i), \SizeOfMaxStabS(G_i)\}
    \text{ for all } i=1,\dots,m.
  \end{equation*}
  In this case, since $G$ is the disjoint union of its connected components, and
  since regularity is additive, with~\eqref{equation:union:maxStable} we have
  \begin{equation*}
    \reg(S/\bei_{G}) = \sum_{i} \reg(S/\bei_{G_i}) \leq \prod_i \NumOfMaxStabS(G_i) = \NumOfMaxStabS(G).
  \end{equation*}
  Finally, by ~\eqref{equation:union:stable} we have
  \begin{equation*}
    \reg(S/\bei_{G}) = \sum_{i} \reg(S/\bei_{G_i}) \leq \sum_{i} \SizeOfMaxStabS(G_i)
    = \SizeOfMaxStabS(G). \qedhere
  \end{equation*}
\end{proof}
The method of bounding by $\NumOfMaxStabS(G)$ seems coarse as the maximum and
sum over a set of intergers are, respectively, replaced by the sum and the
product over those integers.  Nevertheless $s(G)$ can be a good bound as
discussed in Remark~\ref{r:bestBound}.

Since independent vertices cannot be in the same clique,
$\SizeOfMaxStabS(G) \le \NumOfMaxCl(G)$.  Thus
$\reg(S/\bei_{G}) \leq \NumOfMaxCl(G)$ holds for cographs, but in fact it holds
for all graphs by~\cite{malayeri2021proof}.

Our last bound uses the maximum vertex degree $\delta(G)$ of a connected
cograph~$G$.
\begin{proposition}\label{proposition:cographs:regMaxdeg}
  Let $G$ be a connected cograph. Then
  \begin{equation*}
    \reg(S/\bei_{G}) \leq \delta(G) = \max\{\delta(v) : v \in V(G)\}.
  \end{equation*}
\end{proposition}
\begin{proof}
  If $G = K_{n}$, then $\reg(S/\bei_{G}) \le 1$, $\delta(G) = n-1$, and the
  inequality holds.  If $G$ is not complete but connected, it is the join of
  induced subgraphs $G_1,\dots,G_m$ as in Lemma~\ref{lemma:cographs:join}.  Then
  $\reg(S/\bei_{G}) = \max\{2, \reg(S/\bei_{G_{1}}), \dots, \reg(S/\bei_{G_{m}})\}$ by
  Proposition~\ref{proposition:cographs:regConnected}.  Writing $n_{i}$ for the
  number of vertices of $G_{i}$, Theorem~\ref{theorem:matsudaMurai} gives
  \begin{equation*}
    \reg(S/\bei_{G}) \leq \max\{2, n_1-1, \dots, n_m-1\}.
  \end{equation*}
  Let $n_{\text{max}} = \max\{n_1,\dots,n_m\}$.
  Then $\reg(S/\bei_{G}) \leq n_{\text{max}}$, since $G$ is not complete and
  thus $n_{\text{max}} \geq 2$.  Since $G$ is a join of the $G_{i}$, the maximum
  vertex degree satisfies $n_{\text{max}}\leq \max\{\delta(v) : v \in V(G)\}$
  and we conclude.
\end{proof}

\begin{remark}
  The bound in Proposition~\ref{proposition:cographs:regMaxdeg} does not give
  anything new for cone graphs since in this case it agrees with
  Theorem~\ref{theorem:matsudaMurai}.
\end{remark}

\begin{remark}\label{r:bestBound}
  One can ask if one of the bounds in this section is generally preferable over
  the other bounds.  Table~\ref{table:boundComparison} shows than any bound can
  beat any other bound, with the exception of
  $\SizeOfMaxStabS(G) \le \NumOfMaxCl(G)$.  On the other hand, if one asks for
  the best bound, it can be confirmed that among the $2341$ cographs in our
  database, for 505 $\NumOfMaxStabS(G)$ is strictly the best bound and for 724
  $\SizeOfMaxStabS(G)$ is strictly the best bound.  No other bound is ever
  strictly the best and for all remaining graphs there is a tie for the best
  bound.
\end{remark}

\begin{table}[htp]
  \centering
  \begin{tabular}{r | r r r r r}
    & order bound & $\NumOfMaxCl(G)$ & $\NumOfMaxStabS(G)$ & $\SizeOfMaxStabS(G)$ & max deg \\
    \hline
    order bound & $0$ & $968$ & $968$ & $146$ & $1090$ \\
    $\NumOfMaxCl(G)$ & $918$ & $0$ & $1049$ & $0$ & $724$ \\
    $\NumOfMaxStabS(G)$ & $920$ & $1050$ & $0$ & $514$ & $837$ \\
    $\SizeOfMaxStabS(G)$ & $1830$ & $1139$ & $1522$ & $0$ & $1150$ \\
    max deg & $5$ & $362$ & $201$ & $1$ & $0$
  \end{tabular}
  \caption{ \label{table:boundComparison} Comparisons of five regularity bounds
  for all $2341$ cographs in our database.  The numbers given are the numbers of
  cographs for which the bound in the left-most column is strictly better than
  the bound in the top row respectively. `Order bound' stands for the bound in
  Theorem~\ref{theorem:cographs:maxReg}, `max deg' denotes the maximum vertex
  degree in Proposition~\ref{proposition:cographs:regMaxdeg}. Comparisons with
  `max deg' are made only for the $1171$ connected cographs.}
\end{table}

\begin{remark}
  The questions about regularity in this paper can also be asked about the
  regularity of $S/\mathcal{I}_{G}$ where $\mathcal{I}_{G}$ is the \emph{parity
  binomial edge ideal} of~\cite{kahle2015parity}.  Using our database we
  observed the following inequality, slightly weaker than
  Theorem~\ref{theorem:matsudaMurai},
  \[
    \ell \le \reg(S/\mathcal{I}_{G}) \le n.
  \]
  Based on our computations we conjecture that the maximum regularity is
  achieved exactly for disjoint unions of odd cycles.  Minimal free resolutions
  of parity binomial edge ideals contain many interesting patterns that remain
  to be investigated.  At the time of the first posting of this paper,
  explaining even the minimal free resolution of $S/\mathcal{I}_{K_{n}}$ was
  open and we conjectured that $\reg(S/\mathcal{I}_{K_{n}}) = 3$.  In the
  meantime this has been confirmed in~\cite{hoang20:_hilber}.
\end{remark}

\section{Regularity versus \texorpdfstring{$h$}{h}-polynomials}\label{s:HibiMatsuda}

As a standard graded $\kk$-algebra, the Hilbert series of $S/\bei_{G}$ takes the
form $\frac{h_{G}(t)}{(1-t)^{d}}$ where $d$ is the Krull dimension and
$h_{G}(t) \in \mathbb{Z}[t]$.  The numerator $h_{G}$ is known as the
\emph{h-polynomial}.  In the first arXiv version of \cite{HibiMatsuda2018}, Hibi
and Matsuda conjectured that for binomial edge ideals its degree bounds the
regularity from above.  The conjecture was removed from a subsequent version of
their paper after we informed them of the following minimal counterexample on
eight vertices:

\begin{example}
  Let $G$ be the graph in Figure~\ref{figure:regHPolDeg:counterexample}, that is
  the graph on the vertex set $\{1,\dots, 8\}$ with edges $\{1,8\}$, $\{2,6\}$,
  $\{3,7\}$, $\{3,8\}$, $\{4,5\}$, $\{4,8\}$, $\{5,6\}$, $\{5,7\}$, $\{6,7\}$,
  $\{6,8\}$, $\{7,8\}$.  Then $\reg(S/\bei_{G}) = 4$ and $\deg(h_{G}) = 3$.
\end{example}

\begin{figure}[ht]
	\centering
	\begin{tikzpicture}
	[
	scale=0.02,mynode/.style={draw,fill=white,circle,outer sep=4pt,inner sep=2pt},myedge/.style={line width=1.5,black}
	]
	
	\coordinate(p1) at(100,0);
	\coordinate(p2) at(70,70);
	\coordinate(p3) at(0,100);
	\coordinate(p4) at(-70,70);
	\coordinate(p5) at(-100,0);
	\coordinate(p6) at(-70,-70);
	\coordinate(p7) at(0,-100);
	\coordinate(p8) at(70,-70);
	
	\draw[myedge] (p1) -- (p8);
	\draw[myedge] (p2) -- (p6);
	\draw[myedge] (p3) -- (p7);
	\draw[myedge] (p3) -- (p8);
	\draw[myedge] (p4) -- (p5);
	\draw[myedge] (p4) -- (p8);
	\draw[myedge] (p5) -- (p6);
	\draw[myedge] (p5) -- (p7);
	\draw[myedge] (p6) -- (p7);
	\draw[myedge] (p6) -- (p8);
	\draw[myedge] (p7) -- (p8);
	
	\node[mynode] at (p1) {1};
	\node[mynode] at (p2) {2};
	\node[mynode] at (p3) {3};
	\node[mynode] at (p4) {4};
	\node[mynode] at (p5) {5};
	\node[mynode] at (p6) {6};
	\node[mynode] at (p7) {7};
	\node[mynode] at (p8) {8};
	\end{tikzpicture}
	\caption{\label{figure:regHPolDeg:counterexample} A graph with
	$\reg (S/\bei_{G}) > \deg (h_{G})$.}
\end{figure}

At the time of writing, our database contains $39$ counterexamples and none
shows a difference greater than 1 between $\reg(S/\bei_{G})$ and $\deg(h_{G})$.
However, gluing two copies of the counterexample in
Figure~\ref{figure:regHPolDeg:counterexample} at vertex $1$ yields a graph $G$
(visible in Figure~\ref{figure:regHPolDeg:counterexample2}) which satisfies
$\reg(S/\bei_{G})=8$ and $\deg(h_{G}) = 6$.  We now show that the difference can
be made arbitrarily large.  To this end we employ the following two theorems
that explain the behaviour of the regularity and the Hilbert series upon gluing
two graphs $G_{1}$ and $G_{2}$ over a vertex which is a free vertex in both
graphs.  If $G$ is a gluing like this, then $G_{1}$ and $G_{2}$ are a
\emph{split} of~$G$.

\begin{theorem}[{\cite[Theorem~3.1]{JayanthanNarayananRao2016}}]\label{theorem:glue:reg}
  Let $G_1$ and $G_2$ be a split of a graph $G$ at a vertex~$v$. If $v$ is a
  free vertex in both $G_1$ and $G_2$, then
  $\reg(S/\bei_{G}) = \reg(S/\bei_{G_1}) + \reg(S/\bei_{G_2})$.
\end{theorem}

\begin{theorem}[{\cite[Theorem~3.2]{KumarSarkar2018}}]\label{theorem:glue:hilbert}
  Let $G_1$ and $G_2$ be the decomposition of a graph~$G$ at a vertex~$v$. If
  $v$ is a free vertex in both $G_1$ and $G_2$, then
  \begin{equation*}
    \Hilbert_{S/\bei_G}(t) = (1-t)^2 \Hilbert_{S/\bei_{G_1}}(t) \Hilbert_{S/\bei_{G_2}}(t).
  \end{equation*}
\end{theorem}

\begin{theorem}\label{thm:differenceLarge}
  Let $k \in \N$. Then there exists a graph $G$ such that
  \begin{equation*}
    \reg(S/\bei_{G}) = \deg(h_{G}) + k.
  \end{equation*}
\end{theorem}
\begin{proof}
  Let $G_1$ be the graph in Figure~\ref{figure:regHPolDeg:counterexample}.  The
  reduced Hilbert series of $S/\bei_{G_1}$ can be computed with
  \textsc{Macaulay2}~\cite{Macaulay2} as
  \begin{equation*}
    \Hilbert_{S/\bei_{G_1}}(t) = \frac{1+7t+17t^2+13t^3}{(1-t)^9}
  \end{equation*}
  and its regularity as $\reg(S/\bei_{G_1}) = 4$.  Since $G_{1}$ has two free
  vertices $1,2$, we can glue a chain of $k$ copies of $G_{1}$ along free
  vertices (see Figure~\ref{figure:regHPolDeg:counterexample2} for the case
  $k=2$ in which the vertices $2$ and $9$ are available for further gluing).  By
  Theorem~\ref{theorem:glue:hilbert}, the Hilbert series of the resulting graph
  is
  \begin{equation*}
    \Hilbert_{S/\bei_{G_2}}(t) = \dfrac{(1+7t+17t^2+13t^3)^k}{(1-t)^{7k+2}}
  \end{equation*}
  and, by Theorem~\ref{theorem:glue:reg}, $\reg(S/\bei_{G_2}) = 4k$.  Thus
  $\reg(S/\bei_{G}) - \deg(h_{G}) = k$.
\end{proof}

\begin{figure}[ht]
\centering
\begin{tikzpicture}
[
scale=0.02,mynode/.style={draw,fill=white,circle,outer sep=4pt,inner sep=0pt, minimum size=14pt},myedge/.style={line width=1.5,black}
]

\coordinate(p1) at(100,0);
\coordinate(p2) at(70,70);
\coordinate(p3) at(0,100);
\coordinate(p4) at(-70,70);
\coordinate(p5) at(-100,0);
\coordinate(p6) at(-70,-70);
\coordinate(p7) at(0,-100);
\coordinate(p8) at(70,-70);
\coordinate(p9) at(130,70);
\coordinate(p10) at(200,100);
\coordinate(p11) at(270,70);
\coordinate(p12) at(300,0);
\coordinate(p13) at(270,-70);
\coordinate(p14) at(200,-100);
\coordinate(p15) at(130,-70);

\draw[myedge] (p1) -- (p8);
\draw[myedge] (p2) -- (p6);
\draw[myedge] (p3) -- (p7);
\draw[myedge] (p3) -- (p8);
\draw[myedge] (p4) -- (p5);
\draw[myedge] (p4) -- (p8);
\draw[myedge] (p5) -- (p6);
\draw[myedge] (p5) -- (p7);
\draw[myedge] (p6) -- (p7);
\draw[myedge] (p6) -- (p8);
\draw[myedge] (p7) -- (p8);
\draw[myedge] (p1) -- (p15);
\draw[myedge] (p9) -- (p13);
\draw[myedge] (p10) -- (p14);
\draw[myedge] (p10) -- (p15);
\draw[myedge] (p11) -- (p12);
\draw[myedge] (p11) -- (p15);
\draw[myedge] (p12) -- (p13);
\draw[myedge] (p12) -- (p14);
\draw[myedge] (p13) -- (p14);
\draw[myedge] (p13) -- (p15);
\draw[myedge] (p14) -- (p15);

\node[mynode] at (p1) {1};
\node[mynode] at (p2) {2};
\node[mynode] at (p3) {3};
\node[mynode] at (p4) {4};
\node[mynode] at (p5) {5};
\node[mynode] at (p6) {6};
\node[mynode] at (p7) {7};
\node[mynode] at (p8) {8};
\node[mynode] at (p9) {9};
\node[mynode] at (p10) {10};
\node[mynode] at (p11) {11};
\node[mynode] at (p12) {12};
\node[mynode] at (p13) {13};
\node[mynode] at (p14) {14};
\node[mynode] at (p15) {15};
\end{tikzpicture}
\caption{\label{figure:regHPolDeg:counterexample2} A graph with
$\reg (S/\bei_{G}) > \deg (h_{G}) + 1$.}
\end{figure}

\bibliographystyle{amsplain}
\bibliography{bibliography}
\end{document}